\documentclass[a4paper, 11pt]{article}
\usepackage{amsmath,amssymb,esint,amscd,xspace,fancyhdr,color,authblk,srcltx,fontenc,bbm}
\usepackage{hyperref}
\usepackage{cite}
\newtheorem{theorem}{Theorem}[section]

\newtheorem{definition}[theorem]{Definition}
\newtheorem{lemma}[theorem]{Lemma}

\newtheorem{remark}[theorem]{Remark}
\newenvironment{proof}[1][Proof]{\textbf{#1.} }{\hfill\rule{0.5em}{0.5em}}
{\catcode`\@=11\global\let\AddToReset=\@addtoreset
\AddToReset{equation}{section}

\AddToReset{theorem}{section}

\title{\textsc{Harnack inequality for double-phase functionals with Muckenhoupt-type growth functions}}

\author{Minh-Phuong Tran\thanks{Applied Analysis Research Group, Faculty of Mathematics and Statistics, Ton Duc Thang University, Ho Chi Minh City, Vietnam; \texttt{tranminhphuong@tdtu.edu.vn}}, Thanh-Nhan Nguyen\footnote{Corresponding author} \thanks{Group of Analysis and Applied Mathematics, Department of Mathematics, Ho Chi Minh City University of Education, Ho Chi Minh City, Vietnam; \texttt{nhannt@hcmue.edu.vn}}}

\date{\today}

\begin{document}
 
\maketitle
\begin{abstract}

We investigate a general class of variational integrals under a structural condition imposed on the double-phase function, recently introduced in~\cite{ADKO2026}. In this setting, the strong Harnack inequality for non-negative local quasi-minimizers is established via an appropriate De Giorgi-type iteration argument. Most notably, the proposed analytical approach in this paper provides a new perspective for deriving Harnack-type inequalities for more general variational functionals under a Muckenhoupt-type structural condition on the double-phase function, without relying on the classical coefficient-freezing strategy based on the H\"older continuity of the modulating coefficient and the balance condition on the growth exponents.\\


\medskip

\end{abstract}
%\nopagebreak
\maketitle

%\tableofcontents
%-------------------------------------------------------------
%\newpage
\section{Introduction}
\label{sec:intro}

This paper is devoted to the study of local minimizers for a general class of variational integrals of the type
\begin{align}
\label{eq-main}
v \mapsto \mathbb{F}(v,\Omega) := \int_{\Omega} \mathcal{F}(x,v,\nabla v) dx,
\end{align}
where $\Omega \subset \mathbb{R}^n$ be an open bounded subset ($n \ge 2$); the energy density $\mathcal{F}: \Omega \times \mathbb{R} \times \mathbb{R}^n \to [0,\infty)$ is a Carath\'eodory function satisfying the following growth condition
\begin{align}\label{cond-F}
C_{\mathcal{F}}^{-1} \mathcal{H}(x,|\xi|) \le \mathcal{F}(x,z,\xi) \le C_{\mathcal{F}} \mathcal{H}(x,|\xi|), \quad (x,z,\xi) \in \Omega \times \mathbb{R} \times \mathbb{R}^n,
\end{align}
for some constant $C_{\mathcal{F}} \ge 1$. Here, $\mathcal{H}: \mathbb{R}^n \times [0,\infty) \to [0,\infty)$ stands for the double-phase integrand, defined by
\begin{align}\label{def-Hxt}
\mathcal{H}(x,t) = t^p + a(x) t^{q}, \quad (x,t) \in \mathbb{R}^n \times [0,\infty),
\end{align}
where $1<p \le q <\infty$ are the growth exponents and $a: \mathbb{R}^n \to [0,\infty)$ acts as the modulating coefficient. This integrand naturally switches between two polynomial-growth phases, according to the zero set of the coefficient $a(\cdot)$. 

In the last few decades, the study of non-uniform ellipticity and unbalanced integral functionals has received considerable attention from different points of view. Among them, the class of double-phase functionals, also referred to as variational integrals with $(p,q)$-growth conditions, stands out as a fundamental and prototypical model in this field. The structure of these functionals, which dictates the transition between two distinct elliptic phases, naturally gives rise to non-standard behavior, most notably the Lavrentiev phenomenon (see~\cite{Zhikov1995, Zhikov1997}).  Motivated by these analytical challenges, significant research efforts have been dedicated to double-phase problems, leading to a variety of remarkable contributions concerning existence, regularity, qualitative properties of minimizers and solutions, as well as related variational phenomena.

This analysis started with the pioneering works of Marcellini~\cite{Marcellini1989, Marcellini1991} on functionals with non-standard growth conditions. Subsequent developments focused on the interplay between the H\"older continuity rate of the coefficient $a(\cdot)$ and the ratio $q/p$, which plays a decisive role in higher integrability, regularity, and the presence of the Lavrentiev phenomenon, addressed in~\cite{ELM2004, FMM2004, CM2015, CM2015_2, BaCoMin2016, BaCoMin2018}. In these works, the authors obtained sharp regularity for minimizers of double-phase functionals under the critical assumption
\begin{align}
\label{eq:pqa}
\quad a(\cdot) \in C^{0,\alpha}(\Omega), \ \text{for} \ \alpha \in (0,1], \ \text{and} \ \, \frac{q}{p} \le 1 + \frac{\alpha}{n}.
\end{align}
More precisely, they proved that when~\eqref{eq:pqa} holds, local minimizers enjoy locally H\"older continuous gradients, in particular, $\nabla u \in C^{0,\beta}_{\text{loc}}(\Omega)$ for some $\beta \in (0,1)$. Inspired by this framework, the subsequent regularity theory for double-phase problems has generated a rich and rapidly expanding literature, for instance, one could mention the very interesting papers~\cite{ BBO2021, Byun2017Cava, CM2016, FO2019, FM2019, FP2019, FM2021, F2022, FM2023, FM2023_ARMA, FP2024, FMM2004, ELM2004, FS1999}. Although this list of references is far from complete, it clearly demonstrates that this has been a highly active area of research in recent years. 

%Besides the regularity theory for minimizers, extensive work has been dedicated to establishing structural assumptions in various double-phase settings, with particular emphasis on the regularity of modulating coefficients $a(\cdot)$, growth exponents $p, q$, and the absence of the Lavrentiev phenomenon. In this direction, Colombo and Mingione~\cite{CM2015_2} established regularity for bounded minimizers under the assumption $a \in C^{0,\alpha}$ together with the balanced condition $q \le p+\alpha$. This result was subsequently extended in~\cite{BGS2022}, where the H\"older continuity assumption $a \in C^{0,\alpha}(\overline{\Omega})$ was combined with the relaxed condition $q \le p+\alpha \max\left\{1,\frac{p}{n}\right\}$. More recently, Borowski~\cite{Borowski2025} considered higher-order H\"older coefficients $a \in C^{k,\alpha}$, yielding the corresponding generalized gap condition $q \le p+(k+\alpha) \max\left\{1,\frac{p}{n}\right\}$. We also refer the reader to the survey~\cite{MR2021} for a comprehensive overview and related results. 

Aiming to relax the standard H\"older continuity assumption~\eqref{eq:pqa} and extend to more general classes of functionals, Adamadze, Diening, Kopaliani, and Ok~\cite{ADKO2026} introduced a so-called ``Muckenhoupt-type condition'' directly into a class of double-phase integrands. In a parallel direction concerning weighted settings, our recent work~\cite{NT26} established the regularity theory that allows the modulating coefficient $a(\cdot)$ to belong to an appropriate class of Muckenhoupt weights. These developments considerably enlarge the regularity theory for non-uniformly elliptic variational problems beyond the classical framework \eqref{eq:pqa} typically assumed in the existing literature. 

To give the reader a feeling for the idea behind this new structural assumption, let us now be a bit more precise and discuss the technique developed in \cite{ADKO2026}. In this work, the authors introduce the Muckenhoupt-type condition on the integrand $\mathcal{H}$ as follows:
\begin{align}
\label{main-cond}
\|\chi_B\|_{L^{\mathcal{H}}(\mathbb{R}^n)} \|\chi_B\|_{L^{\mathcal{H}^*}(\mathbb{R}^n)} \le C_{\mathcal{H}} |B|, \mbox{ for every ball } B \subset \mathbb{R}^n.
\end{align}
In this condition, $\chi_B$ denotes the indicator function of the ball $B \subset \mathbb{R}^n$, and $|B|$ represents its Lebesgue measure. The notations $L^{\mathcal{H}}(\mathbb{R}^n)$ and $L^{\mathcal{H}^*}(\mathbb{R}^n)$ stand for the generalized Musielak-Orlicz spaces associated with the double-phase function $\mathcal{H}$ and its conjugate $\mathcal{H}^*$, respectively, equipped with the standard Luxemburg norms (see below for precise definitions and notation). Here, the constant $C_{\mathcal{H}} > 0$ is a structural constant depending only on $\mathcal{H}$ and the dimension~$n$. A notable feature of this structural condition is its ability to generalize the classical assumption \eqref{eq:pqa} imposed on the modulating coefficient. In \cite[Example 2.11]{ADKO2026}, the authors present an example clearly demonstrating that the assumption \eqref{main-cond} is not restricted to H\"older continuous coefficients. Instead, it admits a larger class of coefficients, including those generated by natural operations, such as pointwise minima or maxima, under which H\"older continuity is generally lost. Therefore, condition \eqref{main-cond} should be viewed as an extended structural assumption that covers the traditional H\"older setting and a wide range of coefficients falling outside the existing double-phase regularity theory.

Despite its considerable flexibility, the Muckenhoupt-type structural assumption is not without certain limitations. It should be noted that the fundamental inequalities developed in this framework remain valid only under suitable constraints, such as normalization or coupling conditions. For instance, the generalized Jensen inequality holds only for functions $f \in L^{\mathcal{H}}(\mathbb{R}^n)$ satisfying $\|f\|_{L^{\mathcal{H}}(\mathbb{R}^n)} \le 1$ (see \cite[Corollary 2.16]{ADKO2026}); likewise, the Sobolev-Poincar\'e inequality is valid in $W^{1,\mathcal{H}}(B)$ only provided that $\int_{B} \mathcal{H}(x,|\nabla f|) dx \le 1$ (see \cite[Theorem 2.23]{ADKO2026}). These constraints arise quite naturally from the Muckenhoupt-type structure, and one cannot expect the corresponding regularity results to be established on the whole function spaces without imposing additional normalization conditions. As a result, the derivation of local boundedness and H\"older continuity requires a series of intricate technical arguments. For this reason, developing regularity theory under this Muckenhoupt-type structural condition is substantially more challenging if one attempts to follow the classical proof strategy. In particular, proving the Harnack inequality, a key ingredient in establishing the oscillation decay and hence the H\"older continuity of local minimizers, becomes considerably more involved.

This paper establishes the Harnack inequality for non-negative local minimizers of the double-phase functional~\eqref{eq-main}. Furthermore, our analysis is carried out assuming that the integrand $\mathcal{H}$ satisfies \eqref{main-cond}. To put our contribution into perspective, we recall that the classical approach initiated by Baroni, Colombo, and Mingione~\cite{BaCoMin2015} and subsequently employed by Byun and Oh~\cite{BO20}, the Harnack inequality for non-negative local minimizers is derived via the coefficient-freezing technique. However, this strategy relies essentially on the H\"older continuity of the modulating coefficient $a$, and therefore cannot be applied directly under more general structural assumptions. As such, in this setting, it requires a fundamentally different approach and a novel proof strategy. In this work, we developed a new technique based on the self-improving property of the integral average of the double-phase function $\mathcal{H}$ over local balls. Specifically, this method is primarily developed to overcome the difficulties in deriving the supremum estimate, one of the two fundamental ingredients in the proof of the strong Harnack inequality. The methodology addressed in this paper provides a new perspective for establishing Harnack-type inequalities for more general variational functionals under Muckenhoupt-type structural conditions, without relying on the classical coefficient-freezing strategy based on the H\"older continuity of the modulating coefficient and the balance condition on the growth exponents. Furthermore, we expect that the ideas developed in this work can be further adapted to investigate the regularity theory of a larger class of non-uniformly elliptic variational problems. 

For the reader's convenience, let us include here the definitions of the relevant function spaces. Given $k \ge 1$, the Musielak-Orlicz space $L^{\mathcal{H}}(\Omega,\mathbb{R}^k)$, is defined by the set of all measurable functions $f: \Omega \to \mathbb{R}^k$ such that $\int_\Omega \mathcal{H}(x, |f(x)|) dx$ is finite, and equipped with the following Luxemburg norm
\begin{align}\notag
\|f\|_{L^{\mathcal{H}}(\Omega,\mathbb{R}^k)} := \inf \left\{ \lambda > 0: \ \int_\Omega \mathcal{H}\big(x, \lambda^{-1}|f(x)|\big) dx \le 1 \right\}.
\end{align}
Specifically, when $k = 1$, we simply write $L^{\mathcal{H}}(\Omega) \equiv L^{\mathcal{H}}(\Omega,\mathbb{R})$. To clarify the structural condition in~\eqref{main-cond}, let us also briefly introduce the conjugate space. We denote by $\mathcal{H}^*$ the conjugate Young function of $\mathcal{H}$, defined as
\begin{align*}
\mathcal{H}^*(x, t) := \sup_{s \ge 0} \{st - \mathcal{H}(x, s)\}, \quad \text{for } x \in \Omega \ \text{ and } t \ge 0.
\end{align*}
The conjugate Musielak-Orlicz space  $L^{\mathcal{H}^*}(\mathbb{R}^n)$ and its associated Luxemburg norm $\|\cdot\|_{L^{\mathcal{H}^*}(\mathbb{R}^n)}$ are defined analogously, with $\mathcal{H}^*$ replacing $\mathcal{H}$ in the above definitions. Consequently, condition~\eqref{main-cond} can be naturally viewed as a Muckenhoupt-type balance condition formulated in terms of a pair of conjugate Musielak-Orlicz spaces. This condition effectively ensures the validity of a generalized H\"older-type inequality, which is a fundamental tool for establishing our estimates in the Musielak-Orlicz setting.

The Musielak-Orlicz-Sobolev space $W^{1,\mathcal{H}}(\Omega)$ is defined by
\begin{align*}
W^{1,\mathcal{H}}(\Omega) = \left\{f \in L^{\mathcal{H}}(\Omega): \ \nabla f \in L^{\mathcal{H}}(\Omega,\mathbb{R}^n)\right\},
\end{align*}
endowed with the norm $\|f\|_{W^{1,\mathcal{H}}(\Omega)} = \|f\|_{L^{\mathcal{H}}(\Omega)} + \|\nabla f\|_{L^{\mathcal{H}}(\Omega, \mathbb{R}^n)}.$ We further $W_0^{1,\mathcal{H}}(\Omega)$ the space $W_0^{1,1}(\Omega) \cap W^{1,\mathcal{H}}(\Omega)$, equipped with the same norm. For later use, throughout this paper, we also introduce the notation
\begin{align}\label{def-bbL}
\mathbb{L}^{\mathcal{H}}(\Omega) & := \left\{f \in L^{\mathcal{H}}(\Omega): \ \|f\|_{\mathcal{H}(\Omega)} \le 1 \right\}, \\
\mathbb{W}^{\mathcal{H}}(\Omega) & := \left\{w \in W^{1,\mathcal{H}}(\Omega): \ \int_{\Omega} \mathcal{H}(x,|\nabla w|) dx \le 1 \right\}.\label{def-bbW}
\end{align}

With the functional space $W^{1, \mathcal{H}}(\Omega)$ at hand, we are now in a position to define the notion of local minimizers associated with the functional in~\eqref{eq-main}. It is worth noting that we deal here with a class of $\mathcal{Q}$-minimizers, a generalization of standard minimizers considered in the literature, as the latter can be recovered by setting $\mathcal{Q} = 1$. 
\begin{definition}[$\mathcal{Q}$-minimizer]
\label{def:Qminimizer}
A function $u \in W^{1,\mathcal{H}}(\Omega)$ is called a local $\mathcal{Q}$-minimizer of the variational integral $\mathbb{F}$ defined in~\eqref{eq-main}, for some constant $\mathcal{Q} \ge 1$, if the following variational formula
\begin{align}\label{var-form}
\mathbb{F}(u,\mathrm{supp}(u-w)) \le \mathcal{Q}\mathbb{F}(w,\mathrm{supp}(u-w)),
\end{align}
holds for all $w \in W_{\mathrm{loc}}^{1,1}(\Omega)$ and $\mathrm{supp}(u-w) \subset \Omega$.
\end{definition}

In the interest of brevity, for any $\lambda \ge 0$ and a measurable function $f: \Omega \to \mathbb{R}$, we denote the level set $\{x\in \Omega:  f(x) > \lambda\}$ simply as $\Omega \cap \{f>\lambda\}$. Furthermore, we use $\mathtt{data}$ to denote a certain set of structural parameters that will affect the constant dependence in our subsequent statements.

Having introduced the functional setting and notation, we are now ready to state the main result of this paper, namely, the Harnack inequality for local $\mathcal{Q}$-minimizers of the functional $\mathbb{F}$ under the Muckenhoupt-type structural assumption ~\eqref{main-cond}. This is a twofold achievement, as it simultaneously constructs the weak Harnack inequality and the local supremum estimate, whose proofs are non-trivial and require new arguments. The precise statement is given in the following theorem.
\begin{theorem}\label{theo-Harnack}
Let $u \in W_{\mathrm{loc}}^{1,\mathcal{H}}(\Omega)\cap L_{\mathrm{loc}}^{\infty}(\Omega)$ be a non-negative local $\mathcal{Q}$-minimizer of the functional $\mathbb{F}$ defined in~\eqref{eq-main} for some $\mathcal{Q} \ge 1$, under the assumption~\eqref{main-cond}. Fix $x_0 \in \Omega$ and let $r \in (0,1]$ be such that $B_{18r}(x_0) \Subset \Omega$. If $u \in \mathbb{W}^{\mathcal{H}}(B_{18r}(x_0))$, (see~\eqref{def-bbW}), then there exists a constant $C = C\big(\mathtt{data}, (a)_{B_{18r}(x_0)},\|u\|_{L^{\infty}(B_{18r}(x_0))}\big) \ge 1$ such that
\begin{equation}\label{ineq-Harnack}
\sup_{B_r(x_0)} u \le C \inf_{B_r(x_0)} u.
\end{equation}
\end{theorem}

The remainder of this paper is organized as follows. Our arguments draw inspiration from preceding work by Baroni, Colombo, and Mingione in~\cite{BaCoMin2015}, combined with fundamental inequalities established in~\cite{ADKO2026}. The next section is devoted to recalling some known results coupled with the structural assumption~\eqref{main-cond} that are necessary in our subsequent analysis. In Section~\ref{sec:main}, we present the proof of Theorem~\ref{theo-Harnack}.  

\section{Preliminary estimates}
\label{sec:pre}

In this section, we collect some necessary and pivotal estimates that will be used throughout the paper. We begin with two fundamental lemmas regarding the scaling properties of the double-phase function $\mathcal{H}$, whose proofs are verified by direct calculations. 
\begin{lemma}\label{lem-H}
For all $x \in \mathbb{R}^n$ and $\lambda, t \ge 0$, the following inequality holds:
\begin{align}\label{ineq-H-1}
\min\{\lambda^p; \lambda^{q}\} \mathcal{H}(x,t) \le \mathcal{H}(x, \lambda t) \le \max\{\lambda^p; \lambda^{q}\} \mathcal{H}(x,t).
\end{align}
Furthermore, let 
\begin{align}\label{def-PHI}
\mathcal{H}_B(t) = \fint_B \mathcal{H}(x,t)dx, \quad t \in [0,\infty),
\end{align}
denote the integral average of $\mathcal{H}$ over a ball $B \subset \Omega$. Then, $\mathcal{H}_B$ inherits the same scaling property:
\begin{align}\notag
\min\{\lambda^p; \lambda^{q}\} \mathcal{H}_B(t) \le \mathcal{H}_B(\lambda t) \le \max\{\lambda^p; \lambda^{q}\} \mathcal{H}_B(t),
\end{align}
for every $\lambda, t \ge 0$, and the corresponding inverse function $\mathcal{H}_B^{-1}$ satisfies
\begin{align}\label{HB-1}
\min\left\{\lambda^{\frac{1}{p}}; \lambda^{\frac{1}{q}}\right\} \mathcal{H}_B^{-1}(t) \le \mathcal{H}_B^{-1}(\lambda t) \le \max\left\{\lambda^{\frac{1}{p}}; \lambda^{\frac{1}{q}}\right\} \mathcal{H}_B^{-1}(t).
\end{align}
\end{lemma}

The next lemma provides an upper bound for the Luxemburg norm of characteristic functions $\chi_B$ associated with any ball $B \subset \mathbb{R}^n$ in terms of their Lebesgue measure and the averaged modulating coefficient. This estimate, together with the Muckenhoupt-type condition~\eqref{main-cond}, provides the fundamental control of characteristic functions, which will play a crucial role in our setting.
\begin{lemma}\label{lem-norm-chi}
For every ball $B$ in $\mathbb{R}^n$, we have
\begin{align}\label{est-chi-B-norm}
\|\chi_B\|_{L^{\mathcal{H}}(\mathbb{R}^n)} \le  2^{\frac{1}{p}} \max \left\{|B|^{\frac{1}{p}}; \, |B|^{\frac{1}{q}} (a)_B^{\frac{1}{q}} \right\}.
\end{align}
\end{lemma}
\begin{proof}
For every ball $B$ in $\mathbb{R}^n$, we have $\int_B \mathcal{H}\left(x,\lambda_B^{-1} \chi_B(x) \right) dx = 1$, where $\lambda_B = \|\chi_B\|_{L^{\mathcal{H}}(\mathbb{R}^n)}$. Thus, we can estimate
\begin{align}\notag 
1 = |B|\left[\lambda_B^{-p} + \lambda_B^{-q} (a)_B\right] \le 2|B| \max\left\{\lambda_B^{-p}; \lambda_B^{-q} (a)_B\right\},
\end{align}
which directly implies~\eqref{est-chi-B-norm}.
\end{proof}

\begin{remark}
The estimate~\eqref{est-chi-B-norm} highlights the interplay between the geometry of the ball and the double-phase structure of the modular. More precisely, the Luxemburg norm of $\chi_B$ is controlled by both the Lebesgue measure of $B$ and the average of the modulating coefficient, $(a)_B$. It reveals a scale-dependent transition in the energy growth: when the growth is dominated by the $p$-th power, the norm scales by $|B|^{1/p}$, whereas the $q$-th power is modulated via the term $|B|^{1/q} (a)_B^{1/q}$, reflecting the non-uniform ellipticity of the functional $\mathbb{F}$.
\end{remark}

We next recall three fundamental analytical tools established in~\cite{ADKO2026} (see Corollary 2.16, Theorem 2.15, and Theorem 2.23 therein). It is worth noting that although these results were originally formulated for cubes, their validity extends naturally to the case of balls by the same arguments, relying only on the doubling property of the underlying measure, and hence will be employed throughout our analysis.

\begin{lemma}\label{lem-Jensen}
Let $\mathcal{H}$ be the double-phase function defined in~\eqref{def-Hxt}, with $1 <p \le q$, where the modulating coefficient $a: \mathbb{R}^n \to [0,\infty)$ satisfies the Muckenhoupt-type condition~\eqref{main-cond}. Then, there exists a constant $C= C(\mathtt{data}) >0$ such that the following inequality
\begin{align}\label{Jen-ineq}
\fint_B \mathcal{H}\left(x,\fint_B |f(y)|dy\right)dx \le  C \fint_B \mathcal{H} \left(x,|f(x)|\right)dx,
\end{align}
holds for every  $f \in \mathbb{L}^{\mathcal{H}}(\mathbb{R}^n)$ and any ball $B \subset \mathbb{R}^n$, where  $\mathbb{L}^{\mathcal{H}}(\mathbb{R}^n)$ is defined by~\eqref{def-bbL}.
\end{lemma}
\begin{lemma}\label{lem:Re-Holder}
Let $\mathcal{H}$ be the double-phase function defined in~\eqref{def-Hxt}, with $1 <p \le q$, where the modulating coefficient $a: \mathbb{R}^n \to [0,\infty)$ satisfies the Muckenhoupt-type condition~\eqref{main-cond}. Then, there exist constants $\theta = \theta(\mathtt{data})>1$ and $C = C(\mathtt{data})>0$ such that
\begin{align}\label{phi-eps}
\left(\fint_B \big[\mathcal{H} \left(x,t\right)\big]^{\theta}dx\right)^{\frac{1}{\theta}} \le  C \fint_B \mathcal{H}(x,t)dx,
\end{align}
for any ball $B \subset \mathbb{R}^n$ and $0< t \le \left(\|\chi_B\|_{L^{\mathcal{H}}(\mathbb{R}^n)}\right)^{-1}$. 
\end{lemma}

\begin{lemma}\label{lem-Sob-Poin}
Let $\mathcal{H}$ be the double-phase function defined in~\eqref{def-Hxt}, with $1 <p \le q$, where the modulating coefficient $a: \mathbb{R}^n \to [0,\infty)$ satisfies the Muckenhoupt-type condition~\eqref{main-cond}. Let $B_r \Subset \Omega$ for some $r>0$. Then, there exist constants $\theta=\theta(\mathtt{data})>1$ and $C=C(\mathtt{data})>0$ such that for every $u \in \mathbb{W}^{\mathcal{H}}(B_r)$, the following inequality holds
\begin{align}\label{SP-ineq}
\fint_{B_r} \mathcal{H}\left(x,\frac{|u-(u)_{B_r}|}{r}\right)dx \le \left[\fint_{B_r} \mathcal{H}^{\theta}\left(x,\frac{|u-(u)_{B_r}|}{r}\right)dx\right]^{\frac{1}{\theta}} \le C \fint_{B_r} \mathcal{H}\left(x,|\nabla u|\right)dx.
\end{align}
\end{lemma}

\begin{remark}
We emphasize that the combination of the Muckenhoupt condition~\eqref{main-cond} and the doubling property of $\mathcal{H}$ is what ensures the boundedness of $\mathbf{M}$. Then, the boundedness of the maximal operator $\mathbf{M}$ implies the density of smooth functions in $L^{\mathcal{H}}(\Omega)$ and $W^{1,\mathcal{H}}(\Omega)$, which in turn guarantees the absence of the Lavrentiev phenomenon in the corresponding variational setting. The readers may refer to the works~\cite[Section 2.8]{ADKO2026},~\cite[Section 4]{CM2015} for detailed discussions. 
\end{remark}

\section{Proofs of main results}
\label{sec:main}

This section is devoted to the proof of Theorem~\ref{theo-Harnack} under the specific structural assumption~\eqref{main-cond}. Inequality~\eqref{ineq-Harnack} provides a quantitative relation between the local infimum and supremum of non-negative local $\mathcal{Q}$-minimizers of the multi-phase functional $\mathbb{F}$, allowing us to control their oscillation at small scales. The argument is based on a two-step scheme, consisting of a weak Harnack (infimum) estimate and a corresponding supremum estimate, established in Theorems~\ref{theo-weak-Har} and~\ref{theo-sup-Har}. The full Harnack inequality will then follow by a standard combination of these two estimates. 

\subsection{Caccioppoli-type inequality}

We first establish a Caccioppoli-type inequality, which provides a fundamental energy estimate and serves as a key preparatory step for the subsequent arguments. Lemma~\ref{lem-Giuti} is a version of  Giusti's iteration lemma (see, e.g.,~\cite{Giusti2003}), a well-known technical result which plays a central role in deriving quantitative decay estimates. 

\begin{lemma}[Giusti's iteration lemma]
\label{lem-Giuti}
Given $\gamma>0$, $\sigma \in (0,1)$ and $0<r<R$. Let $K: (0,\infty) \to [0,\infty)$ be a function satisfying
\begin{align}\label{Giuti-cond-2}
K(s t) \le s^{-\gamma} K(t), \quad \mbox{ for all } t>0 \mbox{ and } 0<s \le 1.
\end{align}
Assume moreover that a non-decreasing function $G: [r,R] \to [0,\infty)$ satisfies
\begin{align}\label{Giuti-cond}
G(\varrho) \le \sigma G(\rho) + K(\rho-\varrho), \ \mbox{ for all } r \le \varrho < \rho \le R.
\end{align}
Then, there exists a constant $C = C(\gamma, \sigma)>0$ such that
$G(r) \le C K(R - r)$.
\end{lemma}
%\begin{proof}
%Let us fix $\epsilon = \frac{\sigma^{\frac{1}{\gamma}} + 1}{2}$. Since $\sigma \in (0,1)$, one can easily check that $\epsilon \in (0,1)$ and $0< \sigma \epsilon^{-\gamma} <1$. Hence, there exists a constant $C_0 = C_0(\gamma, \sigma)>0$ such that
%$\sum_{j=0}^{\infty} (\sigma \epsilon^{-\gamma})^j \le C_0.$
%We now introduce an increasing sequence $\{t_j\}_{j=0}^{\infty} \subset [r, R]$ defined by $t_0 = r$ and
%\begin{align}\notag
%t_{j+1}= t_j + \epsilon^j (1-\epsilon) (R-r), \quad j \ge 0. 
%\end{align}
%For every $j \ge 0$, substituting $\varrho = t_j$ and $\rho = t_{j+1}$ into~\eqref{Giuti-cond}, and then applying assumption~\eqref{Giuti-cond-2} with $s = \epsilon^j (1-\epsilon)$, it yields
%\begin{align}\notag
%G(t_j) & \le \sigma G(t_{j+1}) + \epsilon^{-j\gamma} (1-\epsilon)^{-\gamma} K\big(R-r\big).\notag
%\end{align}  
%By induction argument, it yields
%\begin{align}\label{Giuti-3}
%G(t_0) \le \sigma^k G(t_{k}) + (1-\epsilon)^{-\gamma} K\big(R-r\big) \sum_{j=0}^{k-1} (\sigma \epsilon^{-\gamma})^j \le \sigma^k G(t_{k}) + C_0(1-\epsilon)^{-\gamma} K\big(R-r\big).
%\end{align}
%Combining the monotonicity of $G$ with~\eqref{Giuti-3}, we arrive at
%\begin{align}\notag
%G(r) \le \sigma^k G(R) + C_0 (1-\epsilon)^{-\gamma} K\big(R-r\big).
%\end{align}
%Since $0<\sigma<1$ and $G(R)<\infty$, letting $k\to \infty$ in both sides, one obtains~\eqref{Giuti-ineq}. The proof is then complete.
%\end{proof}

\begin{lemma}\label{lem-Cacci}
Let $u \in W_{\mathrm{loc}}^{1,\mathcal{H}}(\Omega)$ be a non-negative local $\mathcal{Q}$-minimizer of the functional $\mathbb{F}$ defined in~\eqref{eq-main}, for some $\mathcal{Q} \ge 1$. Let $0<r < R$ be such that $B_R \Subset \Omega$.  Then, there exists a constant $C = C(\mathtt{data})>0$ such that
\begin{align}\label{Cacci-ineq-0}
\int_{B_{r}} \mathcal{H}(x,|\nabla u|) dx & \le C \int_{B_{R}} \mathcal{H}\left(x, \frac{|u - \lambda|}{R - r}\right)dx,\\
\label{Cacci-ineq}
\int_{B_{r}} \mathcal{H}(x,|\nabla (u - \lambda)_{\pm}|) dx & \le C \int_{B_{R}} \mathcal{H}\left(x, \frac{(u - \lambda)_{\pm}}{R - r}\right)dx,
\end{align}
for all $\lambda \in \mathbb{R}$. Here, notations $(\cdot)_{+}$ and $(\cdot)_{-}$ denote the positive and negative parts, respectively, namely, $(u-\lambda)_{+} = \max\{u-\lambda, 0\} \mbox{ and } (u-\lambda)_{-} = \max\{\lambda - u, 0\}.$ 
\end{lemma}
\begin{proof}
Let us fix $0<r < R$ such that $B_R \Subset \Omega$. For every $\varrho$ and $\rho$ such that $0<r \le \varrho < \rho \le R$, let $\eta \in C_0^{\infty}(B_{\rho})$ be a smooth cut-off function such that $0 \le \eta \le 1$, $\eta \equiv 1 \mbox{ on } B_{\varrho}$, and $|\nabla \eta| \le \frac{2}{\rho-\varrho}$. Taking $w = u - \eta (u-\lambda) \in W^{1,\mathcal{H}}(B_{\rho})$, it is easy to check that 
\begin{align}\label{eta-w-set}
|\nabla w| \le (1-\eta)|\nabla u| + 2\frac{|u-\lambda|}{\rho - \varrho} \mbox{ and }  \mathrm{supp} (u - w) \subset B_{\rho}.
\end{align}
Since $u$ is a local $\mathcal{Q}$-minimizer of functional $\mathbb{F}$, by ~\eqref{var-form} and~\eqref{eta-w-set}, one obtains $\mathbb{F}(u,B_{\rho}) \le \mathcal{Q} \mathbb{F}(w,B_{\rho})$. This can be rewritten as
$\int_{B_{\rho}} \mathcal{F}(x,u,\nabla u) dx \le \mathcal{Q} \int_{B_{\rho}} \mathcal{F}(x,w,\nabla w) dx.$
Applying condition~\eqref{cond-F}, this implies that
\begin{align}\notag
\int_{B_{\rho}} \mathcal{H}(x,|\nabla u|) dx \le C_{\mathcal{F}}^2 \mathcal{Q}\int_{B_{\rho}} \mathcal{H}(x,|\nabla w|) dx.
\end{align}
Taking~\eqref{eta-w-set} into account and applying~\eqref{ineq-H-1} in Lemma~\ref{lem-H}, it yields
\begin{align}
\int_{B_{\rho}} \mathcal{H}(x,|\nabla u|) dx & \le  2^{q} C_{\mathcal{F}}^2 \mathcal{Q} \int_{B_{\rho}\setminus B_{\varrho}} \mathcal{H}(x,|\nabla u|) dx  + 4^q C_{\mathcal{F}}^2 \mathcal{Q} \int_{B_{\rho}} \mathcal{H}\left(x, \frac{|u-\lambda|}{\rho - \varrho}\right)dx. \label{Ca-0}
\end{align}
By adding $2^{q} C_{\mathcal{F}}^2 \mathcal{Q} \int_{B_{\varrho}} \mathcal{H}(x,|\nabla u|) dx$ to both sides of~\eqref{Ca-0} and rearranging the terms, one gets
\begin{align}
\int_{B_{\varrho}} \mathcal{H}(x,|\nabla u|) dx 
%& \le \frac{2^{q} C_{\mathcal{F}}^2 \mathcal{Q}}{1+2^{q} C_{\mathcal{F}}^2 \mathcal{Q}} \int_{B_{\rho}} \mathcal{H}(x,|\nabla u|) dx +  \frac{4^{q} C_{\mathcal{F}}^2 \mathcal{Q}}{1+2^{q} C_{\mathcal{F}}^2 \mathcal{Q}} \int_{B_{\rho}} \mathcal{H}\left(x, \frac{|u-\lambda|}{\rho - \varrho}\right)dx \notag \\
& \le \sigma \int_{B_{\rho}} \mathcal{H}(x,|\nabla u|) dx +  2^{q} \int_{B_{R}} \mathcal{H}\left(x, \frac{|u-\lambda|}{\rho - \varrho}\right)dx, \label{Giu-1}
\end{align}
where $\sigma = \frac{2^{q} C_{\mathcal{F}}^2 \mathcal{Q}}{1+2^{q} C_{\mathcal{F}}^2 \mathcal{Q}}<1$. By~\eqref{Giu-1}, it enables us to apply Lemma~\ref{lem-Giuti} with the following two functions
\begin{align*}
G(t) = \int_{B_{t}} \mathcal{H}(x,|\nabla u|) dx,  \ t \in [r,R], 
 \mbox{ and }  K(t) = 2^q\int_{B_{R}} \mathcal{H}\left(x, \frac{|u-\lambda|}{t}\right)dx, \ t>0.
\end{align*}
Here, by using~\eqref{ineq-H-1} in Lemma~\ref{lem-H}, we remark that $K(s t) \le s^{-q} K(t)$ whenever $0<s \le 1$.
Then, one obtains~\eqref{Cacci-ineq-0}. The proof of~\eqref{Cacci-ineq} follows the same argument by choosing the test function $w = u - \eta (u-\lambda)_{\pm}$.
\end{proof}

\subsection{Weak Harnack inequality}

We shall address the weak Harnack inequality for local non-negative $\mathcal{Q}$-minimizers. 
\begin{lemma}\label{lem-LV-1}
Let $u \in W_{\mathrm{loc}}^{1,\mathcal{H}}(\Omega)$ be a nonnegative local $\mathcal{Q}$-minimizer of the functional $\mathbb{F}$ defined in~\eqref{eq-main} , for some $\mathcal{Q} \ge 1$. Let $\lambda >0$, $0<\sigma<2^{-n}$ and $\widetilde r>0$ be such that $B_{4\widetilde{r}}:=B_{4\widetilde r}(x_0)\Subset\Omega$. Assume that $u \in \mathbb{W}^{\mathcal{H}}(B_{4\widetilde{r}})$ and 
\begin{align}\label{asump-Hol}
|B_{\widetilde{r}} \cap \{u >  \lambda\}| \ge \sigma |B_{\widetilde{r}}|.
\end{align}
Then, there exists a constant $\varepsilon=\varepsilon(\mathtt{data}, \sigma) \in (0,\frac{1}{2})$ such that
\begin{align}\label{inf-u-eps}
\inf_{B_{\widetilde{r}}} u \ge \varepsilon \lambda.
\end{align}
\end{lemma}
\begin{proof}
Let $\lambda >0$ and $\sigma \in (0, 2^{-n})$. Setting $\widetilde{\sigma} = 2^{-n}\sigma$, it is clear that $\widetilde{\sigma} \in (0,1)$. Under the assumption~\eqref{asump-Hol}, for $u \in \mathbb{W}^{\mathcal{H}}(B_{4\widetilde{r}})$, it observes that
\begin{align}
|B_{2\widetilde{r}} \cap \{u >  \lambda\}| & \ge |B_{\widetilde{r}} \cap \{u >  \lambda\}|  \ge \sigma |B_{\widetilde{r}}| \ge \widetilde{\sigma} |B_{2\widetilde{r}}|.\notag
\end{align}
At this stage, by applying~\cite[Lemma 3.6]{ADKO2026} with paramater $\tilde{\sigma}$, we conclude that there exists $\varepsilon \in (0,\frac{1}{2})$ such that~\eqref{inf-u-eps} holds. This completes the proof.
\end{proof}

Next, let us invoke the classical covering lemma, originally proposed by Krylov and Safonov~\cite{KS1981}. For the detailed proof, we refer the reader to~\cite{KS2001}.
\begin{lemma}\label{lem-E-Br}
Let $\tilde{\delta} \in (0, 1)$, and let $E \subset B_{2r} \subset \mathbb{R}^n$ be a measurable set. Define the set $E_{\tilde{\delta}}$ as follows:
\begin{align}\notag %\label{E-delta}
E_{\tilde{\delta}} := \bigcup_{\substack{x \in B_{2r}, \, \rho >0}} \left\{ B_{3\rho}(x) \cap B_{2r}: \ |B_{3\rho}(x) \cap E| \ge \tilde{\delta} |B_\rho(x)| \right\}.
\end{align}
Then, either $E_{\tilde{\delta}} = B_{2r}$ or $|E_{\tilde{\delta}}| \ge \frac{1}{2^n \tilde{\delta}} |E|$ holds.
\end{lemma}

\begin{lemma}\label{lem-lambda-0}
Let $u \in W_{\mathrm{loc}}^{1,\mathcal{H}}(\Omega)$ be a nonnegative local $\mathcal{Q}$-minimizer of functional $\mathbb{F}$ defined in~\eqref{eq-main}, with $\mathcal{Q} \ge 1$. Assume that~\eqref{Hu-10r} holds for some $r>0$. Then, there exist constants $C = C(\mathtt{data})>0$ and $\gamma = \gamma(\mathtt{data})>0$ such that
\begin{align}\label{lam-0-est}
\inf_{B_{2r}} u \ge C\left[\frac{|B_{2r} \cap {u > \lambda}|}{|B_{2r}|}\right]^{\gamma} \lambda, \quad \mbox{for every } \lambda>0.
\end{align}
\end{lemma}
\begin{proof}
Let $\mu \in (0,1)$ be a constant depending only on $\mathtt{data}$, whose value is chosen to be sufficiently small and will be specified later. For any fixed $\lambda>0$ and $j \in \mathbb{N}$, we introduce the following level sets
\begin{align}\notag
E := B_{2r} \cap \left\{u > \mu^{j-1}\lambda\right\}  \subset B_{2r}.
\end{align}
Fix $\tilde{\delta}$ small enough such that $0<2^n\tilde{\delta}<2^{-n}$. We define the index set $\mathbb{I}_{\tilde{\delta}}$ as the set all of pairs $(x,\rho)  \in B_{2r} \times (0,4r/3)$ such that 
\begin{align}\label{def-Index}
|B_{3\rho}(x) \cap E| \ge \tilde{\delta} |B_\rho(x)|.
\end{align} 
Here, observe that we only consider $\rho<4r/3$ in the set $\mathbb{I}_{\tilde{\delta}}$ since $B_{3\rho}(x) \supset B_{4r}(x) \supset B_{2r}(x_0) \equiv B_{2r} \supset E$ whenever $\rho \ge 4r/3$.  Let us define the following subset of $B_{2r}$:
\begin{align}\notag
 E_{\tilde{\delta}} := \bigcup_{\substack{(x,\rho) \in \mathbb{I}_{\tilde{\delta}}}} \big[ B_{3\rho}(x) \cap B_{2r}\big].
\end{align}
By the virtue of Lemma~\ref{lem-E-Br}, it ensures that either $E_{\tilde{\delta}} = B_{2r}$ or $|E_{\tilde{\delta}}| \ge (2^n\tilde{\delta})^{-1} |E|$ holds true. Moreover, for each pair $(x,\rho) \in \mathbb{I}_{\tilde{\delta}}$, by~\eqref{def-Index} one has
\begin{align}\notag
|B_{3\rho}(x) \cap \{u > \mu^{j-1}\lambda\}| & \ge |B_{3\rho}(x) \cap E|  \ge \tilde{\delta} |B_{\rho}(x)| \ge 3^{-n}\tilde{\delta} |B_{3\rho}(x)|.
\end{align}
Since $\rho<4r/3$ and $x \in B_{2r}(x_0)$, we observe that $B_{12\rho}(x) \subset B_{16r}(x) \subset B_{18r}(x_0) \equiv B_{18r} \Subset \Omega,$ 
which matches the structural assumption~\eqref{Hu-10r}. Therefore, it allows to apply Lemma~\ref{lem-LV-1} with the ball $B_{\widetilde{r}} \equiv B_{3\rho}(x)$, yielding a constant $\varepsilon_2 = \varepsilon_2(\mathtt{data}) \in (0,1)$ such that 
$$\inf_{B_{3\rho}(x)} u \ge \varepsilon_2 \mu^{j-1} \lambda > \mu^{j} \lambda,$$
provided $0< \mu < \varepsilon_2$. Consequently, we obtain the inclusion $E_{\tilde{\delta}} \subset B_{2r} \cap \{u > \mu^j \lambda\}$. This implies that for each $j \in J_\lambda$, either 
\begin{align}\label{TH1}
B_{2r} \cap \{u > \mu^j \lambda\} = B_{2r}
\end{align}
or
\begin{align}\label{TH2}
|B_{2r} \cap \{u > \mu^j \lambda\}| \ge (2^n\tilde{\delta})^{-1} \left|B_{2r} \cap \left\{u > \mu^{j-1}\lambda\right\}\right|
\end{align}
hold. Let $k(\lambda) \ge 1$ be the integer number such that
\begin{align}\label{m-lam}
(2^n\tilde{\delta})^{k(\lambda)}  < \frac{|B_{2r} \cap \{u > \lambda\}|}{|B_{2r}|} \le (2^n\tilde{\delta})^{k(\lambda)-1}.
\end{align}
We divide the proof into two cases depending on the value of $k(\lambda)$. First, if $k(\lambda) = 1$, that means $|B_{2r} \cap \{u > \lambda\}| > 2^n\tilde{\delta} |B_{2r}|$. Then, Lemma~\ref{lem-LV-1} allows us to find  a constant $\varepsilon_1 = \varepsilon_1(\mathtt{data}) \in (0,1)$ such that 
\begin{align}\label{m-1}
\inf_{B_{2r}} u \ge \varepsilon_1 \lambda > \mu \lambda = \mu^{k(\lambda)} \lambda,
\end{align}
provided $0<\mu<\varepsilon_1$. If  $k(\lambda) > 1$, in this scenario, the index set $J_{\lambda} := \{1,2,...,k(\lambda)-1\}$ is non-empty. If there exists some $j_0 \in J_{\lambda}$ such that~\eqref{TH1} holds, then we immediately get 
\begin{align}\label{m-2}
u > \mu^{j_0} \lambda \ge \mu^{k(\lambda)} \lambda \ \mbox{ a.e. in } B_{2r}.
\end{align}
In the other case,~\eqref{TH2} holds for all $j \in J_{\lambda}$. Thus, by~\eqref{m-lam}, one has
\begin{align}\notag
|B_{2r} \cap \{u > \mu^{k(\lambda)-1} \lambda\}| & \ge (2^n\tilde{\delta})^{-1} \left|B_{2r} \cap \left\{u > \mu^{k(\lambda)-2}\lambda\right\}\right| \\
%& \ge \cdots \notag \\
& \ge (2^n\tilde{\delta})^{-(k(\lambda)-1)} \left|B_{2r} \cap \left\{u > \lambda\right\}\right| \notag \\
& \ge 2^n\tilde{\delta} |B_{2r}|.\notag
\end{align}
Applying Lemma~\ref{lem-LV-1} once again, there is $\varepsilon_3 = \varepsilon_3(\mathtt{data}) \in (0,1)$ such that 
\begin{align}\label{m-3}
\inf_{B_{2r}} u \ge \varepsilon_3 \mu^{k(\lambda)-1} \lambda > \mu^{k(\lambda)} \lambda,
\end{align}
if $0< \mu \le \varepsilon_3$. Taking~\eqref{m-1},~\eqref{m-2} and~\eqref{m-3} into account, it results that
\begin{align}\label{u-m-lambda}
\inf_{B_{2r}} u \ge \mu^{k(\lambda)} \lambda, \ \mbox{ for every } \lambda>0,
\end{align}
where $0<\mu< \min_{1\le i \le 3}\varepsilon_i < 1$. By setting  $\gamma = \log_{2^n\tilde{\delta}} \mu > 0$, inequalities~\eqref{m-lam} and~\eqref{u-m-lambda} yield
\begin{align}\notag
\inf_{B_{2r}} u \ge \mu^{k(\lambda)} \lambda = \left[(2^n\tilde{\delta})^{k(\lambda)}\right]^{\gamma} \lambda \ge \left[2^n\tilde{\delta} \frac{|B_{2r} \cap \{u > \lambda\}|}{|B_{2r}|}\right]^{\gamma} \lambda,
\end{align}
which allows us to conclude~\eqref{lam-0-est}. The proof is complete.
\end{proof}

\bigskip

\begin{theorem}\label{theo-weak-Har}
Let $u \in W_{\mathrm{loc}}^{1,\mathcal{H}}(\Omega)$ be a nonnegative local $\mathcal{Q}$-minimizer of the functional $\mathbb{F}$ defined in~\eqref{eq-main}, for some $\mathcal{Q} \ge 1$. Assume that 
\begin{align}\label{Hu-10r}
B_{18r} \equiv B_{18r}(x_0) \Subset \Omega \mbox{ and } u \in \mathbb{W}^{\mathcal{H}}(B_{18r}) \mbox{ for some } r>0.
\end{align}
There exist an exponent $\alpha_0 = \alpha_0(\mathtt{data}) > 0$ and a constant $C = C(\alpha,\mathtt{data}) \ge 1$, such that the following inequality holds:
\begin{equation}\label{ineq-weak-Harnack}
\left(\fint_{B_{2r}} u^{\alpha} dx \right)^{\frac{1}{\alpha}} \le C \inf_{B_r} u, \quad \mbox{ for every } \alpha \in (0,\alpha_0).
\end{equation}
\end{theorem}
\begin{proof}
Let us define  $\lambda_0:= \inf_{B_{2r}} u$. Thanks to Lemma~\ref{lem-lambda-0}, there exists  $\gamma>0$ such that
\begin{align}\label{lam-0}
\frac{|B_{2r} \cap \{u > \lambda\}|}{|B_{2r}|} \le C \left(\frac{\lambda_0}{\lambda}\right)^{\frac{1}{\gamma}},
\end{align}
for every $\lambda>0$. We now choose $0< \alpha < \frac{1}{\gamma}$ and employ standard layer-cake representation for the $L^\alpha$-norm to get
\begin{align*}
\fint_{B_{2r}} u^{\alpha} dx  & = \int_{0}^{\infty} \alpha \lambda^{\alpha-1} \frac{|B_{2r} \cap \{u > \lambda\}|}{|B_{2r}|} d\lambda,
\end{align*}
one deduces from~\eqref{lam-0} that
\begin{align*}
\fint_{B_{2r}} u^{\alpha} dx  & \le \lambda_0^{\alpha} + \int_{\lambda_0}^{\infty} \alpha\lambda^{\alpha-1} \frac{|B_{2r} \cap \{u > \lambda\}|}{|B_{2r}|} \, d\lambda  \\
& \le \lambda_0^{\alpha} + C \alpha\lambda_0^{\frac{1}{\gamma}} \int_{\lambda_0}^{\infty} \lambda^{\alpha-1-\frac{1}{\gamma}} \, d\lambda = \left(1 + \frac{C \alpha}{\frac{1}{\gamma}-\alpha}\right) \lambda_0^{\alpha}, 
\end{align*}
which establishes the desired weak Harnack inequality~\eqref{ineq-weak-Harnack}. The proof is now complete.
\end{proof}

\subsection{The supremum estimate}

This section establishes the supremum estimates for nonnegative local $\mathcal{Q}$-minimizers. To this aim, we first recall a De Giorgi-type iteration lemma, which is essential for our development. The reader is referred to~\cite[Lemma 5.1]{NT26} for its detailed proof.

\begin{lemma}[De Giorgi-type iteration]
\label{lem:DeGiorgi}
Let $c, \delta, \gamma$, and $\Pi > 1$ be given positive constants. Let $\Psi: [0, \infty) \to [0, \infty)$ be a strictly increasing function satisfying
\begin{align}\notag %\label{assump-beta}
\Psi(\epsilon s) \le \epsilon^{\gamma} \Psi(s) \quad \text{for all } \epsilon \in (0,1] \text{ and } s \ge 0.
\end{align}
Suppose that the non-negative sequence $\{Z_k\}_{k=0}^{\infty}$ satisfies the following conditions
\begin{align}\notag %\label{assump-Ym}
\begin{cases}\Psi(Z_0) &\le c^{-\frac{1}{\delta}} \Pi^{-\frac{1}{\delta}\left(1 + \frac{1}{\delta\gamma}\right)},\\
Z_{k+1} &\le c  \Pi^{k+1} \left[\Psi\big(Z_{k}\big)\right]^{\delta} Z_{k}, \quad k=0,1,2,\dots
\end{cases}
\end{align}
Then, the sequence $\{Z_k\}_{k=0}^{\infty}$ converges to $0$ as $k \to \infty$. 
%Moreover, the following decay estimate holds:
%\begin{align*}\notag %\label{Ym-ansatz}
%Z_{k+1} \le \Pi^{-\frac{k+1}{\delta\gamma}} Z_0, \quad \text{for every } k \ge 0.
%\end{align*}
\end{lemma}

\begin{lemma}\label{lem-sup-Har}
Let $u \in W_{\mathrm{loc}}^{1,\mathcal{H}}(\Omega)\cap L_{\mathrm{loc}}^{\infty}(\Omega)$ be a nonnegative local $\mathcal{Q}$-minimizer of the functional $\mathbb{F}$ defined in~\eqref{eq-main}, for some $\mathcal{Q} \ge 1$. Assume that 
\begin{align}\label{Hu-2r}
B_{2r} \equiv B_{2r}(x_0) \Subset \Omega \mbox{ and } u \in \mathbb{W}^{\mathcal{H}}(B_{2r})
\end{align}
for some $r\in (0,1]$. Then, there exist $\widetilde{\theta} = \widetilde{\theta}(\mathtt{data})>1$ and $C = C(\mathtt{data}) \ge 1$ such that the following inequality
\begin{align}\label{H-x-u-1}
\sup_{B_{\sigma r}} u & \le C (\varpi-\sigma)^{-\widetilde{\theta}} \mathcal{H}_{B_{\varpi r}}^{-1}\left(\fint_{B_{\varpi r}} \mathcal{H}\left(x,u\right) dx\right) 
\end{align}
holds for every pair of radii $\sigma, \varpi$ satisfying $1 \le \sigma < \varpi \le 2$.
\end{lemma}
\begin{proof} 
For any fix $1 \le \sigma < \varpi \le 2$, we first construct a sequence of concentric balls $\{B^k\}_{k=0}^{\infty}$ defined by
\begin{align}\notag
B^k = B_{r_k}, \ \mbox{ with } \ r_k = \left[\sigma + 2^{-k}(\varpi -\sigma)\right]r, \quad \text{ for } k \ge 0. 
\end{align}
Setting $\pi_k = 2^{-k}(\varpi -\sigma) \in (0, 2^{-k})$, it is easily to verify that
\begin{align}\notag
r \le \lim_{j \to \infty} r_j = \sigma r < r_{k+1}  < r_{k} \le r_0 = \varpi r \le 2r, \mbox{ and } \ \frac{r_{k}-r_{k+1}}{r} = \pi_{k+1}.
\end{align}
Next, let $\Lambda$ be a positive constant to be clarified later. For each $k \ge 0$, we introduce the truncated functions and their corresponding support sets as follows:
\begin{align}\notag
v_k = (u  - \lambda_k)_{+}, \ \mbox{ and } \ S^k = B^k \cap \left\{v_k>0\right\},
\end{align}
where the climbing threshold levels are given by $\lambda_k := \Lambda\left(1-2^{-k}\right)$. Our primary objective is to establish a De Giorgi-type iteration for the sequence of energy integrals $\{U_k\}_{k=0}^{\infty}$ defined by
\begin{align}\label{def-Uk}
U_k := \fint_{B^k} \mathcal{H}\left(x,v_k\right)dx, \quad k \ge 0.
\end{align}
Applying H\"older's inequality with the exponent $\theta>1$ specified in Lemma~\ref{lem-Sob-Poin}, we obtain
\begin{align}\label{sup-Har-est-1}
U_{k+1} & = \fint_{B^{k+1}} \mathcal{H}\left(x,v_{k+1}\right) \chi_{S^{k+1}}(x) dx \le \left[\fint_{B^{k+1}} \mathcal{H}^\theta \left(x,v_{k+1}\right) dx \right]^{\frac{1}{\theta}} \left[\frac{|S^{k+1}|}{|B^{k+1}|}\right]^{1 - \frac{1}{\theta}}.
\end{align}
To bound the first factor on the right-hand side of~\eqref{sup-Har-est-1}, we introduce a sequence of smooth cut-off functions $\eta_k \in C_c^{\infty}(\widetilde{B}^{k})$ satisfying
\begin{align}\label{def-eta-k}
0\le \eta_k \le 1, \ \eta_k \equiv 1 \mbox{ in } B^{k+1}, \mbox{ and } |\nabla \eta_k| \le \frac{4}{r_{k}-r_{k+1}},
\end{align}
where the concentric ball $\widetilde{B}^{k} := B_{\widetilde{r}_{k}}$, whose radius $\widetilde{r}_{k}$ is the arithmetic mean of $r_{k+1}$ and $r_{k}$, namely, 
$\widetilde{r}_{k} = \frac{1}{2}(r_{k}+r_{k+1})$.
Throughout the subsequent estimates, we record the following scaling properties for our radii, which can be easily verified from the definitions:
\begin{align}\label{r10}
& \frac{r_{k} -r_{k+1}}{r} = \frac{2\big(r_{k} - \widetilde{r}_{k}\big)}{r} = \pi_{k+1}, \ \mbox{ and } \ \frac{\widetilde{r}_{k}}{r_{k} - \widetilde{r}_{k}} \sim \frac{\widetilde{r}_{k}}{r_{k} - {r}_{k+1}} \sim \frac{1}{\pi_{k+1}}.
\end{align}
Next, let us introduce the scaled test function $w_{k+1}$ defined by
$w_{k+1} := \beta \pi_{k+1} \eta_k v_{k+1}$,
where $\beta > 0$. It ensures that $w_{k+1} \in W_0^{1,\mathcal{H}}(\widetilde{B}^k) \cap \mathbb{W}^{\mathcal{H}}(\widetilde{B}^k)$.  Indeed, by the properties of $\eta_k$ in~\eqref{def-eta-k} and relations in~\eqref{r10}, one has $w_{k+1} \in W_0^{1,\mathcal{H}}(\widetilde{B}^{k})$. Furthermore, the gradient of $w_{k+1}$ satisfies:
\begin{align}\label{est-eta}
|\nabla w_{k+1}| & \le \beta \pi_{k+1} \left( |\nabla v_{k+1}|\eta_k + |\nabla \eta_k| v_{k+1}\right) 
%& \le \beta \pi_{k+1} \left(|\nabla v_{k+1}| + \frac{4}{r \pi_{k+1}}v_{k+1}\right) \notag \\ 
\le \beta \pi_{k+1} |\nabla v_{k+1}| + \frac{4\beta}{r} v_{k+1}.
\end{align}
Since $\beta > 0$ is chosen sufficiently small, we can exploit the the $\Delta_2$-condition of $\mathcal{H}$ as formulated in~\eqref{ineq-H-1} and the estimate~\eqref{est-eta} to obtain:
\begin{align*}
\int_{\widetilde{B}^{k}} \mathcal{H}\big(x, |\nabla w_{k+1}|\big)  dx & \le C  \int_{\widetilde{B}^{k}} \mathcal{H}(x, \beta \pi_{k+1} |\nabla v_{k+1}|) dx   + C \int_{\widetilde{B}^{k}} \mathcal{H}\left(x, \frac{4\beta \widetilde{r}_{k}}{r} \frac{v_{k+1}}{\widetilde{r}_{k}}\right)  dx\\
& \le C  (\beta \pi_{k+1} )^p \int_{\widetilde{B}^{k}} \mathcal{H}(x, |\nabla v_{k+1}|) dx  + C \left(\frac{4\beta \widetilde{r}_{k}}{r} \right)^p \int_{\widetilde{B}^{k}} \mathcal{H}\left(x, \frac{v_{k+1}}{\widetilde{r}_{k}}\right)  dx \\
& \le C \beta^p \left[\int_{\widetilde{B}^{k}} \mathcal{H}(x, |\nabla v_{k+1}|) dx + \int_{\widetilde{B}^{k}} \mathcal{H}\left(x, \frac{v_{k+1}}{\widetilde{r}_{k}}\right)  dx\right].
\end{align*}
By invoking the Sobolev-Poincar\'e's inequality~\eqref{SP-ineq} on the last term, we arrive at
\begin{align*}
\int_{\widetilde{B}^{k}} \mathcal{H}\big(x, |\nabla w_{k+1}|\big)  dx & \le C \beta^p \int_{\widetilde{B}^{k}} \mathcal{H}(x, |\nabla v_{k+1}|) dx  \le  C \beta^p \int_{B_{2r}} \mathcal{H}\left(x, |\nabla u|\right)dx  \le 1.
\end{align*}
Next, choosing $\beta>0$ such that $C \beta^p \le 1$ and using the assumption~\eqref{Hu-2r}, we conclude $w_{k+1} \in \mathbb{W}^{\mathcal{H}}(\widetilde{B}^{k})$. This allows us to apply the Sobolev–Poincar\'e inequality~\eqref{SP-ineq} to the function $w_{k+1}$ on the concentric ball $\widetilde{B}^k$. Noting that $1 < \frac{\widetilde{r}_{k}}{r_{k+1}} < \frac{3}{2}$, we apply the scaling property of $\mathcal{H}$ to derive
\begin{align}\label{app:SP-1}
\left[\fint_{B^{k+1}} \mathcal{H}^\theta \left(x,v_{k+1}\right) dx \right]^{\frac{1}{\theta}} & \le \left[\frac{|\widetilde{B}^{k}|}{|B^{k+1}|} \fint_{\widetilde{B}^{k}} \mathcal{H}^\theta \left(x,\eta_k v_{k+1}\right) dx \right]^{\frac{1}{\theta}} \notag\\
& \le \left[\frac{|\widetilde{B}^{k}|}{|B^{k+1}|} \left(\frac{\widetilde{r}_{k}}{\beta \pi_{k+1}}\right)^{q\theta} \fint_{\widetilde{B}^{k}} \mathcal{H}^\theta\left(x, \frac{w_{k+1}}{\widetilde{r}_{k}}\right) dx \right]^{\frac{1}{\theta}} \notag \\
& \le C \left(\frac{\widetilde{r}_{k}}{\beta \pi_{k+1}}\right)^{q} \fint_{\widetilde{B}^{k}} \mathcal{H}\big(x, |\nabla w_{k+1}|\big) dx.
\end{align}
To estimate the integral on the right-hand side, we invoke the Caccioppoli inequality~\eqref{Cacci-ineq}, the gradient bound~\eqref{est-eta}, and the relations in~\eqref{r10}. Since $r_k - \widetilde{r}_k \sim r \pi_{k+1}$, we obtain
\begin{align}\label{app:SP-2}
\fint_{\widetilde{B}^{k}} \mathcal{H}\big(x, |\nabla w_{k+1}|\big) dx  & \le C \fint_{\widetilde{B}^{k}} \mathcal{H}(x, \beta \pi_{k+1}|\nabla v_{k+1}|) dx  + C\fint_{\widetilde{B}^{k}} \mathcal{H}\left(x, \frac{4\beta}{r}v_{k+1}\right) dx \notag \\
& \le C \frac{|B_{k}|}{|\widetilde{B}^{k}|} \fint_{B^{k}} \mathcal{H}\left(x, \frac{\beta \pi_{k+1} v_{k+1}}{r_{k}-\widetilde{r}_{k}}\right) dx + C \frac{|B^{k}|}{|\widetilde{B}^{k}|} \fint_{B^{k}} \mathcal{H}\left(x, \frac{4\beta}{r}v_{k+1}\right) dx \notag \\
& \le  \frac{C}{r^q} \fint_{B^{k}} \mathcal{H}\left(x, v_{k}\right) dx.
\end{align}
Substituting~\eqref{app:SP-2} into~\eqref{app:SP-1} and invoking the relations from~\eqref{r10}, one obtains the following estimate 
\begin{align}\label{app:SP-3}
\left[\fint_{B^{k+1}} \mathcal{H}^\theta\left(x, v_{k+1}\right) dx \right]^{\frac{1}{\theta}} & \le    C \left(\frac{\widetilde{r}_{k}}{r\beta \pi_{k+1}}\right)^{q} \fint_{B^{k}} \mathcal{H}\left(x, v_{k}\right) dx \le \frac{2^{(k+1)q}C}{(\varpi -\sigma)^q} U_{k}.
\end{align}
In this stage, we turn our attention to bounding the measure of the level set in~\eqref{sup-Har-est-1}. For almost every $x \in S^{k+1}$, it is clear that $v_{k}(x) > \lambda_{k+1} - \lambda_{k} =  2^{-(k+1)} \Lambda \chi_{S^{k+1}}(x)$, which implies
\begin{align}\notag
U_{k} \ge \frac{1}{|B^{k}|} \int_{S^{k+1}} \mathcal{H}\left(x,v_{k}\right)dx  \ge \frac{|B^{k+1}|}{|B^{k}|} \fint_{B^{k+1}} \mathcal{H}\left(x,2^{-(k+1)} \Lambda \chi_{S^{k+1}}(x)\right)dx.\notag
\end{align}
Thanks to the Jensen's inequality~\eqref{Jen-ineq} in Lemma~\ref{lem-Jensen}, it yields
\begin{align}\notag
U_{k} \ge C \fint_{B^{k+1}} \mathcal{H}\left(x, 2^{-(k+1)} \Lambda \frac{|S^{k+1}|}{|B^{k+1}|}\right)dx \ge C \mathcal{H}_{B_{\varpi r}}\left(2^{-(k+1)} \Lambda \frac{|S^{k+1}|}{|B^{k+1}|}\right),
\end{align}
where $\mathcal{H}_{B_{\varpi r}}$ is defined by~\eqref{def-PHI}. This inequality implies that
\begin{align}
\frac{|S^{k+1}|}{|B^{k+1}|} \le C 2^{k+1} \Lambda^{-1} \mathcal{H}_{B_{\varpi r}}^{-1}\left(U_{k}\right).\label{app:SP-4}
\end{align}
Substituting~\eqref{app:SP-3} and~\eqref{app:SP-4} into~\eqref{sup-Har-est-1}, one observes
\begin{align}
U_{k+1} & \le \frac{2^{(k+1)q}C}{(\varpi -\sigma)^q}  \left[2^{k+1} \Lambda^{-1} \mathcal{H}_{B_{2r}}^{-1}\left(U_{k}\right) \right]^{1-\frac{1}{\theta}} U_{k} \le \frac{C_0 \Lambda^{-\frac{\theta-1}{\theta}}}{(\varpi-\sigma)^q} \left[ 2^{q+1-\frac{1}{\theta}}\right]^{k+1} \left[\mathcal{H}_{B_{\varpi r}}^{-1}\left(U_{k}\right) \right]^{\frac{\theta- 1}{\theta}} U_{k}. \notag
\end{align}
Applying Lemma~\ref{lem:DeGiorgi} with the following parameters 
\begin{align*}
Z_k = U_k, \ \delta = \frac{\theta-1}{\theta}, \ \gamma = \frac{1}{q}, \ c = \frac{C_0 \Lambda^{-\frac{\theta-1}{\theta}}}{(\varpi -\sigma)^q}, \ \Pi = 2^{q+1-\frac{1}{\theta}}, \ \Psi = \mathcal{H}_{B_{\varpi r}}^{-1},
\end{align*}
it results that $\lim_{k \to \infty} U_k = 0$ provided
\begin{align}
\mathcal{H}_{B_{\varpi r}}^{-1}(U_0) = \left[\frac{C_0 \Lambda^{-\frac{\theta-1}{\theta}}}{(\varpi -\sigma)^q}\right]^{-\frac{\theta}{\theta-1}} \left[2^{q+1-\frac{1}{\theta}}\right]^{-\frac{\theta}{\theta-1}\left(1+\frac{\theta q}{\theta-1}\right)}. \notag
\end{align}
More precisely, we choose $\Lambda$ such that
\begin{align}
\Lambda = C_0^{\frac{\theta}{\theta-1}} 2^{\left(\frac{\theta q}{\theta-1}+1\right)^2} (\varpi-\sigma)^{-\frac{\theta q}{\theta-1}} \mathcal{H}_{B_{\varpi r}}^{-1}(U_0).\notag
\end{align}
It implies $\sup_{B_{\sigma r}} u \le \Lambda$, which deduces 
\begin{align}
\sup_{B_{\sigma r}} u & \le (\varpi-\sigma)^{-\widetilde{\theta}} \mathcal{H}_{B_{\varpi r}}^{-1}(U_0), \ \mbox{ with } \widetilde{\theta} := \frac{\theta q}{\theta-1}. \notag
\end{align}
We arrive at the estimate~\eqref{H-x-u-1} by the definition of $U_0$ in~\eqref{def-Uk}, therefore concluding the proof of Lemma~\ref{lem-sup-Har}. 
\end{proof}

\begin{theorem}\label{theo-sup-Har}
Let $u \in W_{\mathrm{loc}}^{1,\mathcal{H}}(\Omega)\cap L_{\mathrm{loc}}^{\infty}(\Omega)$ be a non-negative local $\mathcal{Q}$-minimizer of the functional $\mathbb{F}$ defined in~\eqref{eq-main}, with $\mathcal{Q} \ge 1$. Assume that 
\begin{align}\label{Hu-2r}
B_{2r} \equiv B_{2r}(x_0) \Subset \Omega \mbox{ and } u \in \mathbb{W}^{\mathcal{H}}(B_{2r})
\end{align}
holds for some $r\in (0,1]$. Then, for every $\alpha \in (0,p)$, there exists a constant $C\ge 1$, depending on $\mathtt{data}$, $\alpha$, $(a)_{B_{2r}}$ and $\|u\|_{L^{\infty}(B_{2r})}$, such that the following inequality holds
\begin{align}
\label{ineq-sup-Harnack}
\sup_{B_r} u \le C \left( \fint_{B_{2r}} u^{\alpha} dx \right)^{\frac{1}{\alpha}}.
\end{align}
\end{theorem}
\begin{proof}
Thanks to Lemma~\ref{lem-sup-Har}, there exist two constants $\widetilde{\theta} = \widetilde{\theta}(\mathtt{data})>1$ and $C = C(\mathtt{data}) \ge 1$ such that~\eqref{H-x-u-1} holds for every $1 \le \sigma < \varpi \le 2$. The key technique in our proof is based on Lemma~\ref{lem:Re-Holder}, which yields the reverse H\"oder property of $\mathcal{H}(\cdot,t)$ over $B_{\varpi r}$ under the constraint $0<t \le \left(\|\chi_{B_{\varpi r}}\|_{L^{\mathcal{H}}(\mathbb{R}^n)}\right)^{-1}$. For this reason, we separate the proof into two cases: 
$$0<\|\chi_{B_{\varpi r}}\|_{L^{\mathcal{H}}(\mathbb{R}^n)} \sup_{B_{\varpi r}} u \le 1 \ \mbox{ and } \ \|\chi_{B_{\varpi r}}\|_{L^{\mathcal{H}}(\mathbb{R}^n)} \sup_{B_{\varpi r}} u > 1.$$ 
In the first case, thanks to inequality~\eqref{phi-eps} in Lemma~\ref{lem:Re-Holder}, there exists a constant $\theta_1 = \theta_1(\mathtt{data})>1$ such that 
\begin{align}
\left(\fint_{B_{\varpi r}} \left[\mathcal{H}\left(x,\sup_{B_{\varpi r}} u\right)\right]^{\theta_1} dx \right)^{\frac{1}{\theta_1}} \le C\fint_{B_{\varpi r}} \mathcal{H}\left(x,\sup_{B_{\varpi r}}  u\right) dx.\label{RH-B-varpi-01}
\end{align}
Since $u \le \sup_{B_{\varpi r}}  u$ in $B_{\varpi r}$, for every $\alpha \in (0,p)$, we have
\begin{align}\notag
\fint_{B_{\varpi r}} \mathcal{H}\left(x,u\right) dx & = \fint_{B_{\varpi r}} \left(\frac{u(x)}{\sup_{B_{\varpi r}} u}\right)^p \left[ \left(\sup_{B_{\varpi r}} u\right)^p + a(x) u^{q-p} \left(\sup_{B_{\varpi r}} u\right)^p \right] dx \\
%& \le \fint_{B_{\varpi r}} \left(\frac{u(x)}{\sup_{B_{\varpi r}} u}\right)^p \mathcal{H}\left(x,\sup_{B_{\varpi r}} u\right) dx \notag \\
& \le \fint_{B_{\varpi r}} \left(\frac{u(x)}{\sup_{B_{\varpi r}}u}\right)^{\alpha} \mathcal{H}\left(x,\sup_{B_{\varpi r}} u\right) dx. \notag
\end{align}
Applying H\"older's inequality and then combining with~\eqref{RH-B-varpi-01}, we obtain
\begin{align}
\fint_{B_{\varpi r}} \mathcal{H}\left(x,u\right) dx & \le \left(\fint_{B_{\varpi r}} \left[\frac{u(x)}{\sup_{B_{\varpi r}} u}\right]^{\frac{\alpha \theta_1}{\theta_1-1}} dx\right)^{\frac{\theta_1-1}{\theta_1}} \left(\fint_{B_{\varpi r}} \left[\mathcal{H}\left(x,\sup_{B_{\varpi r}} u\right)\right]^{\theta_1} dx \right)^{\frac{1}{\theta_1}} \notag\\
%& \le C\left(\fint_{B_{\varpi r}} \left[\frac{u(x)}{\sup_{B_{\varpi r}} u}\right]^{\alpha} dx\right)^{\frac{\theta_1-1}{\theta_1}} \fint_{B_{\varpi r}} \mathcal{H}\left(x,\sup_{B_{\varpi r}} u\right) dx \notag\\
& = C\left(\fint_{B_{\varpi r}} \left[\frac{u(x)}{\sup_{B_{\varpi r}} u}\right]^{\alpha} dx\right)^{\frac{\theta_1-1}{\theta_1}} \mathcal{H}_{B_{\varpi r}}\left(\sup_{B_{\varpi r}} u\right),\label{H-u-10}
\end{align}
where, in the second inequality of~\eqref{H-u-10}, we use the fact that $\frac{u}{\sup_{B_{\varpi r}} u} \le 1$ in the ball $B_{\varpi r}$.  Next, substituting~\eqref{H-u-10} into the sup-estimate~\eqref{H-x-u-1} and invoking~\eqref{HB-1}, we arrive at:
\begin{align}
\sup_{B_{\sigma r}} u & \le C (\varpi-\sigma)^{-\widetilde{\theta}} \mathcal{H}_{B_{\varpi r}}^{-1}\left(\left(\fint_{B_{\varpi r}} \left[\frac{u(x)}{\sup_{B_{\varpi r}} u}\right]^{\alpha} dx\right)^{\frac{\theta_1-1}{\theta_1}} \mathcal{H}_{B_{\varpi r}}\left(\sup_{B_{\varpi r}}  u\right)\right) \notag \\
& \le C (\varpi-\sigma)^{-\widetilde{\theta}} \left(\fint_{B_{\varpi r}} \left[\frac{u(x)}{\sup_{B_{\varpi r}} u}\right]^{\alpha} dx\right)^{\frac{\theta_1-1}{\theta_1 q}}  \left(\sup_{B_{\varpi r}} u\right) \notag \\
%& \le C (\varpi-\sigma)^{-\widetilde{\theta}} \left(\fint_{B_{\varpi r}} u^{\alpha} dx\right)^{\frac{\theta_1-1}{\theta_1 q}} \left(\sup_{B_{\varpi r}} u\right)^{1-\frac{\alpha(\theta_1-1)}{\theta_1 q}} \notag \\
& \le \left[C (\varpi-\sigma)^{-\frac{\widetilde{\theta} \theta_1 q}{\alpha(\theta_1-1)}}  \left(\fint_{B_{\varpi r}} u^{\alpha} dx\right)^{\frac{1}{\alpha}}\right]^{\frac{\alpha(\theta_1-1)}{\theta_1 q}}  \left(\sup_{B_{\varpi r}}  u\right)^{1-\frac{\alpha(\theta_1-1)}{\theta_1 q}}. \notag
\end{align}
Applying Young's inequality to the last term of the above estimate, we obtain
\begin{align}
\sup_{B_{\sigma r}} u \le \frac{1}{2} \sup_{B_{\varpi r}} u + C (\varpi-\sigma)^{-\gamma} \left(\fint_{B_{\varpi r}} u^{\alpha} dx\right)^{\frac{1}{\alpha}}, \label{H-x-u-12}
\end{align}
where $\gamma := \frac{\widetilde{\theta} \theta_1 q}{\alpha(\theta_1-1)} > 1$. Regarding the second case $\mu := \|\chi_{B_{\varpi r}}\|_{L^\mathcal{H}} \sup_{B_{\varpi r}} u > 1$, we observe that by replacing $u$ with $\mu^{-1}u$, the normalized supremum satisfies $\|\chi_{B_{\varpi r}}\|_{L^\mathcal{H}} \sup (\mu^{-1}u) = 1$. Therefore, we can apply a similar technique to the first case by replacing $u$ with $\mu^{-1} u$. However, some exponents of $\mu$ may appear. Indeed, the reverse H\"older inequality~\eqref{phi-eps} can be applied to the case of $\mu^{-1}u$, yielding:
\begin{align}
\left(\fint_{B_{\varpi r}} \left[\mathcal{H}\left(x,\sup_{B_{\varpi r}} \mu^{-1} u\right)\right]^{\theta_1} dx \right)^{\frac{1}{\theta_1}} \le C\fint_{B_{\varpi r}} \mathcal{H}\left(x,\sup_{B_{\varpi r}} \mu^{-1} u\right) dx.\label{RH-B-varpi}
\end{align}
Since $\mu>1$ and $\mu^{-1} u \le \sup_{B_{\varpi r}} \mu^{-1} u$ in $B_{\varpi r}$, for every $\alpha \in (0,p)$, we have
\begin{align}\notag
\fint_{B_{\varpi r}} \mathcal{H}\left(x,u\right) dx & \le \mu^q \fint_{B_{\varpi r}} \mathcal{H}\left(x,\mu^{-1} u\right) dx 
%& \le \mu^q \fint_{B_{\varpi r}} \left(\frac{\mu^{-1} u(x)}{\sup_{B_{\varpi r}}\mu^{-1} u}\right)^p \mathcal{H}\left(x,\sup_{B_{\varpi r}}\mu^{-1} u\right) dx \notag \\
& \le \mu^q \fint_{B_{\varpi r}} \left(\frac{u(x)}{\sup_{B_{\varpi r}}u}\right)^{\alpha} \mathcal{H}\left(x,\sup_{B_{\varpi r}} \mu^{-1} u\right) dx. \notag
\end{align}
Applying H\"older's inequality and then combining with~\eqref{RH-B-varpi}, we obtain
\begin{align}
\fint_{B_{\varpi r}} \mathcal{H}\left(x,u\right) dx & \le \mu^q \left(\fint_{B_{\varpi r}} \left[\frac{u(x)}{\sup_{B_{\varpi r}} u}\right]^{\frac{\alpha\theta_1}{\theta_1-1}} dx\right)^{\frac{\theta_1-1}{\theta_1}} \left(\fint_{B_{\varpi r}} \left[\mathcal{H}\left(x,\sup_{B_{\varpi r}} \mu^{-1} u\right)\right]^{\theta_1} dx \right)^{\frac{1}{\theta_1}} \notag\\
%& \le C \mu^q \left(\fint_{B_{\varpi r}} \left[\frac{u(x)}{\sup_{B_{\varpi r}} u}\right]^{\alpha} dx\right)^{\frac{\theta_1-1}{\theta_1}} \fint_{B_{\varpi r}} \mathcal{H}\left(x,\sup_{B_{\varpi r}} \mu^{-1} u\right) dx \notag\\
& = C \mu^q \left(\fint_{B_{\varpi r}} \left[\frac{u(x)}{\sup_{B_{\varpi r}} u}\right]^{\alpha} dx\right)^{\frac{\theta_1-1}{\theta_1}}  \mathcal{H}_{B_{\varpi r}}\left(\sup_{B_{\varpi r}} \mu^{-1} u\right).\label{qua-met}
\end{align}
Substituting~\eqref{qua-met} into~\eqref{H-x-u-1}, and recalling that $\mu>1$, inequality~\eqref{HB-1} in Lemma~\ref{lem-H} gives us
\begin{align}
\sup_{B_{\sigma r}} u & \le C (\varpi-\sigma)^{-\widetilde{\theta}} \mathcal{H}_{B_{\varpi r}}^{-1}\left[\mu^q \left(\fint_{B_{\varpi r}} \left[\frac{u(x)}{\sup_{B_{\varpi r}} u}\right]^{\alpha} dx\right)^{\frac{\theta_1-1}{\theta_1}} \mathcal{H}_{B_{\varpi r}}\left(\sup_{B_{\varpi r}} \mu^{-1} u\right)\right] \notag \\
%& \le C (\varpi-\sigma)^{-\widetilde{\theta}} \mu^{\frac{q}{p}} \left(\fint_{B_{\varpi r}} \left[\frac{u(x)}{\sup_{B_{\varpi r}} u}\right]^{\alpha} dx\right)^{\frac{\theta_1-1}{\theta_1 q}}  \left(\sup_{B_{\varpi r}} \mu^{-1} u\right) \notag \\
& \le C (\varpi-\sigma)^{-\widetilde{\theta}} \mu^{\frac{q}{p}-1} \left(\fint_{B_{\varpi r}} u^{\alpha} dx\right)^{\frac{\theta_1-1}{\theta_1 q}}  \left(\sup_{B_{\varpi r}}  u\right)^{1-\frac{\alpha(\theta_1-1)}{\theta_1 q}}. \label{B-sig-mu}
\end{align}
On the other hand, thanks to Lemma~\ref{lem-norm-chi}, one has
\begin{align*}
\|\chi_{B_{\varpi r}}\|_{L^{\mathcal{H}}(\mathbb{R}^n)} & \le  2^{\frac{1}{p}} \max \left\{|{B_{2 r}}|^{\frac{1}{p}}; |{B_{2 r}}|^{\frac{1}{q}}(a)_{B_{2r}}^{\frac{1}{q}}\right\} \le C(\mathtt{data},(a)_{B_{2r}}),
\end{align*}
whenever $r\le 1$. It ensures that $\mu \le C = C(\mathtt{data}, (a)_{B_{2r}},\|u\|_{L^{\infty}(B_{2r})})$. For this reason, inequality~\eqref{B-sig-mu} deduces that
\begin{align}
\sup_{B_{\sigma r}} u 
%& \le C (\varpi-\sigma)^{-\widetilde{\theta}} \left[\left(\fint_{B_{\varpi r}} u^{\alpha} dx\right)^{\frac{1}{\alpha}}\right]^{\frac{\alpha(\theta_1-1)}{\theta_1 q}} \left(\sup_{B_{\varpi r}}  u\right)^{1-\frac{\alpha(\theta_1-1)}{\theta_1 q}} \notag \\
& \le \left[C (\varpi-\sigma)^{-\gamma} \left(\fint_{B_{\varpi r}} u^{\alpha} dx\right)^{\frac{1}{\alpha}}\right]^{\frac{\alpha(\theta_1-1)}{\theta_1 q}} \left(\sup_{B_{\varpi r}}  u\right)^{1-\frac{\alpha(\theta_1-1)}{\theta_1 q}}, \label{H-x-u-2}
\end{align}
where $\gamma$ is determined as in the previous case. Applying Young's inequality on the right-hand side of~\eqref{H-x-u-2}, one also obtains the same formula as in~\eqref{H-x-u-12}. Coupling the estimates for both cases, we conclude that  
\begin{align}
\sup_{B_{\sigma r}} u \le \frac{1}{2} \sup_{B_{\varpi r}} u + C(\varpi-\sigma)^{-\gamma} \left(\fint_{B_{2r}} u^{\alpha} dx\right)^{\frac{1}{\alpha}}, \label{H-x-u-45}
\end{align}
for every $1 \le \sigma < \varpi \le 2$. Finally, we can conclude~\eqref{ineq-sup-Harnack} from~\eqref{H-x-u-45} by applying Lemma~\ref{lem-Giuti}.
% with two following functions
%\begin{align*}
%G(t) = \sup_{B_{t r}} u, \quad t \in [1,2], \ \mbox{ and } \ K(t) = Ct^{-\gamma} \left(\fint_{B_{2r}} u^{\alpha} dx\right)^{\frac{1}{\alpha}}, \quad t >0.
%\end{align*}
The proof is now complete.
\end{proof}

\subsection{Proof of Theorem~\ref{theo-Harnack}}

\begin{proof}[Proof of Theorem~\ref{theo-Harnack}]
The proof of this theorem follows directly from Theorems~\ref{theo-weak-Har} and~\ref{theo-sup-Har}. More precisely, thanks to Theorems~\ref{theo-weak-Har}, one can find $C_1 = C_1(\mathtt{data}) \ge 1$ and $\alpha = \alpha(\mathtt{data}) \in (0,p)$ small enough such that
\begin{equation}\label{last-ineq-1}
\left(\fint_{B_{2r}} u^{\alpha} dx \right)^{\frac{1}{\alpha}} \le C_1 \inf_{B_r} u.
\end{equation}
On the other hand, Theorem~\ref{theo-sup-Har} ensures that
\begin{equation}\label{last-ineq-2}
\sup_{B_r} u \le C_2 \left( \fint_{B_{2r}} u^{\alpha} dx \right)^{\frac{1}{\alpha}},
\end{equation}
for some $C_2 = C_2(\mathtt{data}, (a)_{B_{18r}},\|u\|_{L^{\infty}(B_{18r})}) \ge 1$. Hence, we deduce~\eqref{ineq-Harnack} by combining inequalities~\eqref{last-ineq-1} and~\eqref{last-ineq-2}.
\end{proof}

\section*{Acknowledgement}
This research is funded by Vietnam National Foundation for Science and Technology Development (NAFOSTED), Grant Number: 101.02-2025.03.

\section*{Conflict of Interest}
The authors declared that they have no conflict of interest.

\section*{Declarations}
Data sharing not applicable to this article as no datasets were generated or analysed during the current study.

\end{document}